\documentclass{amsart}

\usepackage{xypic}
\usepackage{amssymb}

\theoremstyle{plain}
\newtheorem{thm}{Theorem}[section]
\newtheorem{cor}[thm]{Corollary}
\newtheorem{lem}[thm]{Lemma}
\newtheorem{prop}[thm]{Proposition}

\newtheorem*{introtheorem}{Theorem} 

\theoremstyle{definition}
\newtheorem{defn}{Definition}[section]
\newtheorem{rem}{Remark}[section]
\newtheorem{example}{Example}[section]

\newcommand{\calB}{{\mathcal B}}
\newcommand{\Ff}{{\mathcal F}}
\newcommand{\Hi}{{\mathcal H}}
\newcommand{\Ll}{{\mathcal L}}
\newcommand{\Or}{{\mathcal O}}

\newcommand{\p}{{\mathfrak p}}
\newcommand{\q}{{\mathfrak q}}

\newcommand{\IC}{{\mathbb C}}
\newcommand{\IF}{{\mathbb F}}
\newcommand{\IP}{{\mathbb P}}

\newcommand{\IZ}{{\mathbb Z}}

\newcommand{\higherlim}[2]{\displaystyle\setbox1=\hbox{\rm lim}
	\setbox2=\hbox to \wd1{\leftarrowfill} \ht2=0pt \dp2=-1pt
	\setbox3=\hbox{$\scriptstyle{#1}$}
	\def\test{#1}\ifx\test\empty
	\mathop{\mathop{\vtop{\baselineskip=5pt\box1\box2}}}\nolimits^{#2}
	\else
	\ifdim\wd1<\wd3
	\mathop{\hphantom{^{#2}}\vtop{\baselineskip=5pt\box1\box2}^{#2}}_{#1}
	\else
	\mathop{\mathop{\vtop{\baselineskip=5pt\box1\box2}}_{#1}}%
	\nolimits^{#2}
	\fi\fi}

\newcommand{\pcom}{{}_{p}^{\wedge}}

\DeclareMathOperator{\Ab}{Ab}
\DeclareMathOperator{\Aut}{Aut}
\DeclareMathOperator{\diag}{diag}
\DeclareMathOperator{\Fred}{Fred}
\DeclareMathOperator{\Hom}{Hom}
\DeclareMathOperator{\im}{Im}
\DeclareMathOperator{\Inj}{Inj}
\DeclareMathOperator{\Inn}{Inn}
\DeclareMathOperator{\Map}{Map}
\DeclareMathOperator{\Mod}{Mod}
\DeclareMathOperator{\Out}{Out}
\DeclareMathOperator{\rad}{rad}
\DeclareMathOperator{\reg}{reg}
\DeclareMathOperator{\Rep}{Rep}
\DeclareMathOperator{\res}{res}
\DeclareMathOperator{\tr}{tr}

\begin{document}
\title[A completion theorem for fusion systems]{A completion theorem for fusion systems}
\author[N. B\'arcenas]{No\'e B\'arcenas}
\address{Centro de Ciencias Matem\'aticas. UNAM\\
Ap.Postal 61-3 Xangari.\\
Morelia, Michoac\'an\\
MEXICO 58089}
\email{barcenas@matmor.unam.mx} 
\author[J. Cantarero]{Jos\'e Cantarero}
\address{Consejo Nacional de Ciencia y Tecnolog\'ia \\
Centro de Investigaci\'on en Matem\'aticas, A.C. Unidad M\'erida \\
Parque Cient\'ifico y Tecnol\'ogico de Yucat\'an \\
Carretera Sierra Papacal-Chuburn\'a Puerto Km 5.5 \\
M\'erida, YUC 97302, Mexico}
\email{cantarero@cimat.mx}


\begin{abstract}
We show that the twisted $K$-theory of the classifying space of a $p$-local finite
group is isomorphic to the completion of the Grothendieck group of twisted representations 
of the fusion system with respect to the augmentation ideal of the representation ring of 
the fusion system. We use this result to compute the $K$-theory of the Ruiz-Viruel 
exotic $7$-local finite groups.
\end{abstract}

\maketitle
% \tableofcontents

\section*{Introduction}

Homotopy invariants of the classifying space of a finite group can often be described in terms of 
algebraic invariants of the group. An important example is Atiyah's completion theorem \cite{A}, 
which shows that the $K$-theory of the classifying space is isomorphic to the completion of the 
representation ring with respect to the augmentation ideal. This result was 
extended to compact Lie groups in \cite{AH} and \cite{AS} and to twisted $K$-theory in \cite{L}.

A $p$-local finite group is an algebraic structure introduced by Broto, Levi and Oliver \cite{BLO}
with the purpose of giving a combinatorial model for the $p$-completion of the classifying 
space of a finite group, but also for other spaces that behave similarly. These additional
$p$-local finite groups are called exotic and there are many examples in the literature by
now, see for instance \cite{BM}, \cite{BLO}, \cite{COS}, \cite{DRV}, \cite{LO}, \cite{O2}, \cite{RV}. 

Each $p$-local finite group $(S,\Ff,\Ll)$ has a classifying space $| \Ll | \pcom$ and some of its homotopy 
invariants have been described in terms of the ``algebraic part'' of the $p$-local 
finite group, which is a saturated fusion system $\Ff$ over a finite $p$-group $S$. 
For example, the fundamental group of $| \Ll | \pcom$ is the quotient of $S$ by the hyperfocal 
subgroup of $\Ff$, and $H^2(|\Ll | \pcom;\IF_p)$ is the set of equivalence classes of central 
extensions of $\Ff$ by $\IZ/p$ (see \cite{BCGLO}). 

In this article we give a description of the twisted $K$-theory of $| \Ll | \pcom$ in terms of algebraic 
invariants of $(S,\Ff)$. Namely, we consider the twistings of $K$-theory classified by the integral third 
cohomology group and in Section \ref{section:twistedrepresentations}, we show that any element in 
$H^3(|\Ll|\pcom;\IZ)$ is represented by a $2$-cocycle of $S$ with values in a cyclic $p$-group. Our main theorem 
is the following

\begin{introtheorem}
Let $(S,\Ff,\Ll)$ be a $p$-local finite group and $\alpha$ a $2$-cocycle of $S$ with values in a
cyclic $p$-group representing an element $\beta_n(\alpha)$ in $H^3(|\Ll|\pcom;\IZ)$. The completion of ${}^{\alpha} R(\Ff)$ 
with respect to the augmentation ideal of $R(\Ff)$ is isomorphic to ${}^{\beta_n(\alpha)} K(| \Ll | \pcom)$.
\end{introtheorem}

Here ${}^{\alpha} R(\Ff)$ is the Grothendieck group of the monoid of $\alpha$-twisted complex representations
of $S$ that are $\Ff$-invariant in a certain sense. For the trivial $2$-cocycle, this corresponds to 
the representation ring $R(\Ff)$ of the fusion system, already studied in \cite{CCM} and \cite{CGT}. 
A complex representation $\rho$ of $S$ is $\Ff$-invariant if $\rho_{|P}$ is equivalent to $\rho_{|f(P)} \circ f$
for any $P \leq S$ and any $ f \colon P \to S$ in $\Ff$.

In Section \ref{section:ideals}, we adapt Atiyah's proof from \cite{A} to determine the prime ideals of $R(\Ff)$ 
and compare the adic topologies on $R(S)$ induced from the augmentation ideals of $R(S)$ and $R(\Ff)$. Both $R(\Ff)$ 
and the $K$-theory of $| \Ll | \pcom$ are computed by stable elements, in the sense that they are limits over the orbit category of $\Ff$-centric subgroups of the representation ring functor and the $K$-theory of the classifying space functor. The completion theorem in this case, which is proved in Section \ref{section:completion}, follows from the naturality of Atiyah's completion map, the results in Section \ref{section:ideals} and the exactness of the completion functor for finitely generated modules. 

To study the general case, we see in Section \ref{section:twistedrepresentations} that an element $\alpha$ in $H^3(| \Ll | \pcom;\IZ)$ 
induces a central extension $(S_{\alpha},\Ff_{\alpha},\Ll_{\alpha})$ of the $p$-local finite group $(S,\Ff,\Ll)$ by $\IZ/p^n$, where 
$p^n$ is the order of $S$. There is a process of turning an $\alpha$-twisted representation of $S$ into a representation of $S_{\alpha}$
and we say that the twisted representation is $\Ff$-invariant when the corresponding twisted representation is $\Ff_{\alpha}$-invariant.
In this way we obtain that ${}^{\alpha} R(\Ff)$ is isomorphic to the Grothendieck group $R(\Ff_{\alpha};\IZ/p^n)$ of $\Ff_{\alpha}$-invariant 
representations of $S_{\alpha}$ where $\IZ/p^n$ acts by scalar multiplication. Moreover, we show that this group can be computed by stable 
elements in $\Ff_{\alpha}$.

In order to compare this group with the twisted $K$-theory of $| \Ll | \pcom$, we note that the
classifying space $|\Ll_{\alpha}| \pcom $ fits into a principal $B\IZ/p^n$-bundle $|\Ll_{\alpha}| \pcom \to | \Ll | \pcom$.
We provide in Section \ref{section:twistedcompletion} an alternative description of twisted $K$-theory 
for spaces $X$ which are the base space of such a principal bundle with respect to the twisting that classifies
the bundle. In particular, the twisted $K$-theory of $| \Ll  | \pcom $ is given as a set of equivariant homotopy 
classes of $B\IZ/p^n$-equivariant maps from $|\Ll_{\alpha}| \pcom$ to a certain space of Fredholm operators. 
The advantage of this characterization is that we can give a stable elements formula in $\Ff_{\alpha}$. The 
completion maps of \cite{L}, seen from this point of view, are natural and this allows us to use the same strategy 
as in the untwisted case to prove the general completion theorem.

In Section \ref{section:final} we use this theorem for several computations. We compute the $K$-theory of the 
$p$-completion of $B\Sigma_p$ for any odd prime $p$ and the twisted $K$-theory of the $2$-completion of $BA_4$. 
We also determine the $K$-theory of the classifying spaces for the three exotic $7$-local finite groups from \cite{RV}. 

\subsection{Acknowledgments}
The  first author  thanks  the  support  of  PAPIIT-UNAM grants  IA 100315 ``Topolog\'ia  Algebraica  y  la  Aplicaci\'on 
de  Ensamble  de  Baum-Connes'', IA 100117 ``Topolog\'ia Equivariante  y Teor\'ia  de  \'indice'', as well  as  CONACYT-SEP  
Foundational  Research  grant ``Geometr\'ia No  conmutativa, Aplicaci\'on de  Baum-Connes  y  Topolog\'ia  Algebraica''. 
The second author was partially supported by SEP-CONACYT grant 242186. We are grateful to the referee for providing very helpful suggestions. This is a post-peer-review, pre-copyedit version of an article published in the Israel Journal of Mathematics. The final authenticated version is available online at: https://doi.org/10.1007/s11856-020-1981-4

\section{Preliminaries}
\label{section:preliminaries}

We first review the models of twisted $K$-theory and some of the results in the theory of $p$-local finite groups 
that will be used in the paper. This material is included for completeness and to fix some of the notation. \newline
 
Twisted $K$-theory with respect to an integral twist $\alpha \in H^3(X;\IZ)$ was first introduced in \cite{DOKA} and \cite{R}. 
We will adopt the homotopy theoretical viewpoint, as stated originally in \cite{ASTW}, with the point-set topology considerations
in \cite{BEJU}.

Let $\Hi$ be an infinite-dimensional separable complex Hilbert space. A projective bundle $P \to X$ is a locally trivial
bundle with fiber $\IP(\Hi)$ and structure group $PU(\Hi)$. We give $PU(\Hi) = U(\Hi)/S^1$ the quotient topology,
where $U(\Hi)$ is given the compactly generated topology generated by the compact-open topology. Since $PU(\Hi)$ is 
a $K(\IZ,2)$, isomorphism classes of projective bundles over $X$ are classified by elements in $H^3(X;\IZ)$. However, a
twisting will be a projective bundle, which corresponds to a $3$-cocycle.

Let $\Fred'(\Hi)$ be the set of pairs $(A,B)$ of bounded operators on $\Hi$ such that $AB-1$ and $BA-1$ are compact operators. 
Endow $\Fred'(\Hi)$ with the topology induced by the embedding
\[ \Fred'(\Hi) \to B(\Hi) \times B(\Hi) \times K(\Hi) \times K(\Hi) \]
\[ (A,B) \mapsto (A,B,AB - 1,BA - 1) \]
where $B(\Hi)$ is the space of bounded operators with the compact-open topology and $K(\Hi)$ is 
the space of compact operators with the norm topology. The group $PU(\Hi)$ acts on $\Fred'(\Hi)$ 
by conjugation. Given a projective bundle $P \to X$, we denote by $\Fred(P)$ the associated bundle 
$ P \times_{PU(\Hi)} \Fred'(\Hi)$ over $X$.

\begin{defn}
The $P$-twisted $K$-theory of $X$ is the group of homotopy classes of sections of $\Fred(P) \to X$.
We denote it by ${}^P K(X)$.
\end{defn} 

Twisted $K$-theory is functorial with respect to maps of spaces and maps of projective bundles.
Namely, a map $ f \colon Y \to X$ induces a group homomorphism $ f^* \colon {}^P K(X) \to {}^{f^*P} K(Y)$.
And a map of projective bundles $ g \colon P \to Q$ over $X$ induces a group homomorphism $ g_* \colon {}^P K(X) \to {}^Q K(X)$.
Hence if two projective bundles are isomorphic, the corresponding twisted $K$-theory groups are isomorphic, but
not canonically. There is also a multiplication 
\[ {}^P K(X) \times {}^Q K(X) \to {}^{P \otimes Q} K(X) \]
In particular, ${}^P K(X)$ is a module over $K(X)$. \newline

%%%%%%%%%%%%%%%%%%%%%%%%%%%%%%%%%%%%%%%%%%%%%%%%%%%%%%

Now we recall the notion of a $p$--local finite group introduced by Broto, Levi and
Oliver in \cite{BLO} as a way of modelling combinatorially the $p$--completion of 
the classifying space of a finite group. Their definition was based on the concept 
of saturated fusion system, first described by Puig \cite{P}.

Given subgroups $P$ and $Q$ of $S$ we denote by $\Hom_S(P,Q)$ the set of group homomorphisms $P \to Q$ that
are conjugations by an element of $S$ and by $\Inj(P,Q)$ the set of monomorphisms from $P$ to $Q$.

\begin{defn}
A fusion system $\Ff$ over a finite $p$--group $S$ is a subcategory of the category of groups 
whose objects are the subgroups of $S$ and such that the set of morphisms $\Hom_{\Ff}(P,Q)$ between two 
subgroups $P$ and $Q$ satisfies the following conditions:
\begin{enumerate}
    \item[(a)] $\Hom_S (P,Q) \subseteq \Hom_{\Ff}(P,Q) \subseteq \Inj(P,Q)$ for all $P,Q \leq S$.
    \item[(b)] Every morphism in $\Ff$ factors as an isomorphism in $\Ff$ followed by an inclusion.
\end{enumerate}
\end{defn}
 
\begin{defn}
Let $\Ff$ be a fusion system over a $p$--group $S$.
\begin{itemize}
    \item We say that two subgroups $P,Q \leq S$ are $\Ff$--conjugate if they are isomorphic in $\Ff$.
    \item A subgroup $P\leq S$ is fully centralized in $\Ff$ if $|C_S (P)| \geq |C_S(P')|$ for all 
    $P' \leq S$ which are $\Ff$--conjugate to $P$.
    \item A subgroup $P\leq S$ is fully normalized in $\Ff$ if $|N_S (P)| \geq |N_S(P')|$ for all 
    $P' \leq S$ which are $\Ff$--conjugate to $P$.
    \item $\Ff$ is a saturated fusion system if the following conditions hold:
    \begin{enumerate}
        \item [(I)] Each fully normalized subgroup $P \leq S$ is fully centralized and
        the group $\Aut_S(P)$ is a $p$--Sylow subgroup of $\Aut_{\Ff}(P)$.
        \item [(II)] If $P \leq S$ and $\varphi \in \Hom _{\Ff}(P,S)$ are such that
        $\varphi P$ is fully centralized, and if we set
\[ N_\varphi = \{ g\in N_S(P) \mid \varphi c_g \varphi^{-1} \in \Aut_S(\varphi P) \}, \]
        then there is $\overline{\varphi} \in \Hom_{\Ff}(N_\varphi ,S)$ such that $\overline{\varphi}_{|P} =\varphi$.
    \end{enumerate}
\end{itemize}
\end{defn}

The motivating example for this definition is the fusion system of a finite group $G$. For a fixed Sylow 
$p$--subgroup $S$ of $G$, let $\Ff_S(G)$ be the fusion system over $S$ defined by setting $\Hom_{\Ff_S(G)}(P,Q)=\Hom_G(P,Q)$.
This is a saturated fusion system. Examples of other exotic examples can be found for instance in \cite{BM}, \cite{BLO}, 
\cite{COS}, \cite{DRV}, \cite{LO}, \cite{O2}, \cite{RV}.

Alperin's fusion theorem for saturated fusion systems (Theorem A.10 in \cite{BLO}) shows that morphisms can be expressed
as compositions of inclusions and automorphisms of a particular class of subgroups, called $\Ff$-centric radical subgroups. 
This result is often used to reduce limits over the orbit category of $\Ff$ to the full subcategory whose objects are 
$\Ff$-centric subgroups or $\Ff$-centric radical subgroups.

\begin{defn}
Let $\Ff$ be a fusion system over a $p$--group $S$.
\begin{itemize}
\item A subgroup $P \leq S$ is $\Ff$--centric if $P$ and all its $\Ff$--conjugates contain their $S$--centralizers.
\item A subgroup $P \leq S$ is $\Ff$--radical if $\Aut_{\Ff}(P)/\Inn(P)$ is $p$--reduced, that is, it has no
nontrivial normal $p$--subgroup.
\end{itemize}
\end{defn}

The notion of a centric linking system is the extra structure needed in the definition of a $p$--local
finite group to obtain a classifying space which behaves like $BG \pcom$ for a finite group $G$. For the
purposes of this paper, it suffices to say that a centric linking system $\Ll$ associated to a saturated
fusion system is category $\Ll$ whose objects are the $\Ff$-centric subgroups of $S$, together with a
functor $\pi \colon \Ll \to \Ff^c$ that satisfies certain properties. See Definition 1.7 in \cite{BLO}.

\begin{defn}
A $p$--local finite group is a triple $(S,\Ff,\Ll)$, where $\Ff$ 
is a saturated fusion system over a finite $p$--group $S$ and 
$\Ll$ is a centric linking system associated to $\Ff$. The 
classifying space of the $p$--local finite group $(S,\Ff,\Ll)$ 
is the space $|\Ll| \pcom$, the $p$-completion in the sense of 
Bousfield-Kan \cite{BK} of the geometric realization of the 
category $\Ll$.
\end{defn}

A theorem of Chermak \cite{C} (see also Oliver \cite{O}) states that every saturated fusion system admits a centric linking system  
and that it is unique up to isomorphism of centric linking systems. In particular, the classifying space of a $p$--local
finite group is determined up to homotopy equivalence by the saturated fusion system. In this paper we need to consider cohomological
decompositions of $|\Ll| \pcom$ over the orbit category. The orbit category of $\Ff$ is the category whose objects are the subgroups 
of $S$ and whose morphisms are
\[ \Hom_{\Or(\Ff)}(P,Q)= \Rep_{\Ff}(P,Q) := \Hom_{\Ff}(P,Q) /\Inn(Q) . \]
Also, $\Or(\Ff^c)$ is the full subcategory of $\Or(\Ff)$ whose objects are the $\Ff$--centric subgroups of
$S$. Theorem 4.2 of \cite{CCM} shows that cohomological invariants of $p$--local finite groups can be computed by stable elements, 
in the following sense.

\begin{thm} 
Let $h^*$ be a generalized cohomology theory.  Given a $p$--local finite group $(S,\Ff,\Ll)$, there is an isomorphism
\[ h^*(|\Ll |\pcom) \cong \higherlim{\Or(\Ff^c)}{} h^*(BP) . \]
\end{thm}

Finally, we would like to say a few words about the theory of central extensions of $p$--local finite groups in \cite{BCGLO}. The
center of a fusion system $\Ff$ over $S$ is the subgroup
\[ Z_{\Ff}(S) = \higherlim{\Ff^c}{} Z \]
where $Z$ is the functor $ P \mapsto Z(P)$. If $A$ is a subgroup of $Z_{\Ff}(S)$ and $(S,\Ff,\Ll)$ is a $p$--local finite group,
there is a $p$--local finite group $(S/A,\Ff/A,(\Ll/A)^c)$ with
\[ \Hom_{\Ff/A}(P/A,Q/A) = \im[\Hom_{\Ff}(P,Q) \to \Hom(P/A,Q/A)] \]
and such that there is a principal fibration
\[ BA \to |\Ll| \pcom \to |(\Ll/A)^c| \pcom \]

A central extension of a $p$--local finite group $(S,\Ff,\Ll)$ by a finite abelian $p$--group $A$ is a $p$-local finite group 
$(\widetilde{S},\widetilde{\Ff},\widetilde{\Ll})$ such that $A \leq Z_{\tilde{\Ff}}(S)$ and $(S,\Ff,\Ll)$ is isomorphic to $(\widetilde{S}/A,\widetilde{\Ff}/A,(\widetilde{\Ll}/A)^c) $. The following theorem is part of Theorem 6.13 in \cite{BCGLO}

\begin{thm}
There is a bijective correspondence between 
\begin{itemize}
\item Equivalence classes of central extensions of $(S,\Ff,\Ll)$ by $A$.
\item Equivalence classes of principal fibrations $BA \to X \to |\Ll| \pcom$.
\item Elements of $H^2(|\Ll| \pcom ;A)$.
\end{itemize}
\end{thm}

\section{The representation ring of a fusion system}
\label{section:ideals}

In this section we determine the prime ideals of the 
representation ring of a saturated fusion system $\Ff$ 
over a finite $p$-group $S$. The argument is inspired
by Section 6 of \cite{A}. We use this information to 
compare different topologies on representation rings 
of subgroups of $S$. We start by recalling the definition of $\Ff$-invariant
representations from \cite{CCM}.

\begin{defn}
\label{defrepresentationfusionsystem}
Let $\Ff$ be a saturated fusion system over a finite $p$-group $S$. A representation $\rho \colon S \to U_n$
is $\Ff$-invariant if the representations $\rho_{|P}$ and $\rho_{|f(P)} \circ f$ are isomorphic for any $P \leq S$ 
and any $ f \in \Hom_{\Ff}(P,S)$. We will also refer to $\rho$ as a representation of $\Ff$.
\end{defn}

A representation $\rho$ is $\Ff$-invariant if and only if its character $\chi_{\rho}$ is $\Ff$-invariant
in the sense that $ \chi_{\rho}(s) = \chi_{\rho}(f(s))$ for any morphism $f$ in $\Ff$ (See Remark 3.3 in \cite{CCM}). 
In particular, the regular representation of $S$ and the trivial representation are $\Ff$-invariant for 
any saturated fusion system $\Ff$ over $S$. Note that the set $\Rep(\Ff)$ of isomorphism classes of unitary representations
of $S$ which are $\Ff$-invariant is a monoid under the operation induced by direct sum of representations. It is closed under tensor product as well.

\begin{defn}
The representation ring $R(\Ff)$ of $\Ff$ is the Grothendieck group
of $\Rep(\Ff)$.
\end{defn}

By Proposition 5.7 in \cite{CCM}, it can be described by stable elements:
\[ R(\Ff) \cong \higherlim{\Or(\Ff^c)}{} R(Q) \]
Note that if $V$ is an $\Ff$-invariant representation which is isomorphic to a direct sum $ A \oplus B$ of representations
of $S$, where $A$ is $\Ff$-invariant, then so is $B$. Hence any $\Ff$-invariant representation is isomorphic
to a direct sum of $\Ff$-invariant representations which can not be decomposed as direct sums of nontrivial
$\Ff$-invariant subrepresentations. We will call these irreducible $\Ff$-invariant representations. Note that 
an irreducible $\Ff$-invariant representation may be reducible as a representation of $S$, and that the
decomposition of an $\Ff$-invariant representation as a direct sum of irreducible $\Ff$-invariant representations 
is not necessarily unique.

Let $X_{\Ff}$ be the set of $\Ff$-conjugacy classes of elements of $S$. Taking an $\Ff$-invariant 
representation to its character induces a ring homomorphism $ \chi_{\Ff} \colon R(\Ff) \otimes \IC \to \Map(X_{\Ff},\IC)$.

\begin{lem}
\label{CharacterIsomorphism}
The map $ \chi_{\Ff} \colon R(\Ff) \otimes \IC \to \Map(X_{\Ff},\IC) $
is an isomorphism.
\end{lem}

\begin{proof}
Note that $\Rep(\Ff)$ is a submonoid of $\Rep(S)$ by definition,
hence the forgetful map $R(\Ff) \to R(S)$ is an injective ring
homomorphism. Since $\IC$ is flat over $\IZ$, the induced map 
$R(\Ff) \otimes \IC \to R(S) \otimes \IC$ is also injective. 

Let $X_S$ the set of $S$-conjugacy classes of elements of $S$ and consider 
the following commutative diagram
\[
\diagram
R(\Ff) \otimes \IC \rto^{\chi_{\Ff}} \dto^j & \Map(X_{\Ff},\IC) \dto^{r^*} \\
R(S) \otimes \IC \rto^{\chi_S} & \Map(X_S,\IC)
\enddiagram
\]
where the horizontal maps are induced by sending representations to their characters
and $r \colon X_S \to X_{\Ff}$ is the quotient map.

It is well known that $\chi_S$ is an isomorphism, hence $\chi_{\Ff}$ is injective. 
Given a map $ f \colon X_{\Ff} \to \IC$, we have $ r^*(f) = \chi_S(z)$ for some 
$z \in R(S) \otimes \IC$. But this implies
\[ z \in \higherlim{\Or(\Ff^c)}{} (R(P) \otimes \IC) \]
And since $\IC$ is flat over $\IZ$
\[ \higherlim{\Or(\Ff^c)}{} (R(P) \otimes \IC) = \left( \higherlim{\Or(\Ff^c)}{} R(P) \right) \otimes \IC = \im(j) \]
And then $\chi_{\Ff}(z) = f$, that is, $\chi_{\Ff}$ is an isomorphism.
\end{proof}

Since $R(S)$ is a free abelian group and $R(\Ff)$ is a subgroup of $R(S)$, it is
also free abelian.

\begin{cor}
The rank of the free abelian group $R(\Ff)$ equals the number of $\Ff$-conjugacy
classes of $S$.
\end{cor}

We denote by $I(G)$ the augmentation ideal of $R(G)$, 
that is, the kernel of the augmentation map $R(G) \to \IZ$.
Similarly, we denote by $I(\Ff)$ the kernel of the augmentation map $R(\Ff) \to \IZ$. 

Given $P \leq S$, we can regard $R(P)$ as a module over $R(\Ff)$ via the composition
$R(\Ff) \to R(S) \to R(P)$. In this section we prove that the $I(P)$-adic topology of 
$R(P)$ coincides with the $I(\Ff)$-adic topology of $R(P)$ seen as
an $R(\Ff)$-module via this map. By Theorem 6.1 in \cite{A},
it suffices to show that the $I(S)$-adic topology of $R(S)$ 
coincides with its $I(\Ff)$-adic topology.

Let $n$ be the order of $S$ and consider the subring $A$ of $\IC$
generated by $\IZ$ and the $n$th roots of unity. Characters of 
representations of $S$ take values in $A$, since they are sums 
of eigenvalues of unitary matrices whose multiplicative orders 
divide $n$. Since $A$ is flat over $\IZ$, a similar argument 
to the one used in the proof of Lemma \ref{CharacterIsomorphism}
shows that $R_A(\Ff):= R(\Ff) \otimes A$ has a decomposition
\[ R_A(\Ff) = \higherlim{\Or(\Ff^c)}{} R_A(P) \]
and it can be seen as a subring of $\Map(X_{\Ff},A)$.

\begin{lem}
\label{PrimeIdeals}
Every prime ideal of $R_A(\Ff)$ is the restriction of a prime ideal of $\Map(X_{\Ff},A)$
and every prime ideal of $R(\Ff)$ is the restriction of a prime ideal of $R_A(\Ff)$.
\end{lem}

\begin{proof}
Since $\Map(X_{\Ff},A)$ is a finitely generated abelian group, it follows by 
the theorem of Cohen-Seidenberg (Theorem 3 in page 257 of \cite{SZ}) as in 
Lemma 6.2 of \cite{A}.
\end{proof}

Note that $\Map(X_{\Ff},A)$ is a finitely generated free $\IZ$-module, in particular
it is a Noetherian ring. Given $x \in X_{\Ff}$ and a prime ideal $\p$ of $A$, the set of 
maps $f$ with $f(x) \in \p$ is a prime ideal of $\Map(X_{\Ff},A)$ and all its prime 
ideals are of this form. Let us denote
\[ P_{\p,x} = \{ \chi \in R_A(\Ff) \mid \chi(x) \in \p \} \]
where $\chi(x)$ denotes the image of any element of $x$. By the previous lemma, all 
the prime ideals of $R_A(\Ff)$ have this form.

If $\p \neq (0)$ then $\p \cap \IZ$ must be a nontrivial prime ideal of $\IZ$, hence
$\p \cap \IZ = q(\p)\IZ$ for a certain prime number $q(\p)$. Moreover, $\p$ is a maximal ideal of $A$ and 
$ A/\p$ is a finite field of characteristic $q(\p)$. 

\begin{lem}
\label{CharacterRegular}
If $\p \neq (0)$ and $q(\p)=p$, then $\chi(s) \equiv \chi(1) \mod \p $ for any \mbox{$\chi \in R_A(\Ff)$} and any $s \in S$.
\end{lem}

\begin{proof}
Note that $R_A(\Ff)$ is a subring of $R_A(S)$, hence we regard $\chi \in R_A(S)$.
Since the order of $s$ is a power of $p$, the proof of Lemma 6.3 in \cite{A} shows
that $\chi(s) \equiv \chi(1) \mod \p$.
\end{proof}

For simplicity, we will use the following notation. Given $x \in X_{\Ff}$, 
we will denote
\[ x_{\p} = \left\{ \begin{array}{ll}
                     1 & \text{ if } \p \neq (0) \text{ and } q(\p) = p \\
                     x & \text{ otherwise} \end{array} \right. \]
where $1$ denotes the $\Ff$-conjugacy class of $1$. With this notation and by
the previous lemma, we have $\chi(x) \equiv \chi(x_{\p}) \mod \p$ for any 
$\chi \in R_A(\Ff)$, any $x \in S$ and any prime ideal $\p$ of $A$.

\begin{lem}
The ideal $P_{\p,x}$ is contained in $P_{\q,y}$ if and only if $\p \subseteq \q$
and $ x_{\p} = y_{\q}$.
\end{lem}

\begin{proof}
Suppose that $\p \subseteq \q$ and $x_{\p} = y_{\q}$. Let $\chi \in P_{\p,x}$, that is,
$\chi(x) \in \p$. Then we have $\chi(x_{\p}) \equiv \chi(x) \mod \p$ and so $\chi(x_{\p})
= \chi(y_{\q})$ belongs to $\p$. But $\chi(y) \equiv \chi(y_{\q}) \mod \q$, hence $\chi(y)$
belongs to $\q$. Therefore $\chi \in P_{\q,y}$.

On the other hand, $P_{\p,x} \subseteq P_{\q,y}$ implies $ \p = P_{\p,x} \cap A \subseteq P_{\q,y} \cap A = \q$.
Suppose $x_{\p} \neq y_{\q}$. The class $y_{\q}$ is the disjoint union of a finite number of $S$-conjugacy classes,
say of elements $y_1, \ldots, y_n$.

Assume $\q \neq (0)$ and $q(\q)=q$. By Lemma 3 in \cite{BT} there is a character $\chi_i$ in $R_A(S)$
that takes values in $\IZ$ and such that $\chi_i(y_i) \equiv 1 \mod q$ and $\chi_i(z)=0$ if 
$z$ is not $S$-conjugate to $y_i$. Let $a_i = \chi_i(y_i)$ and consider the character
\[ \chi = \sum_{i=1}^n \left( \prod_{j \neq i} a_j \right) \chi_i \]
which again lies in $R_A(S)$. This character also takes values in $\IZ$ and it satisfies 
\[ \chi(y_j) = \prod_{i=1}^n a_i \equiv 1 \mod q \]
for all $j$. Moreover $\chi(z) = 0$ if $z$ is not $S$-conjugate to $y_i$ for any $i$. Since this character is constant
on each $\Ff$-conjugacy class, it actually lies in $R_A(\Ff)$. And it satisfies $\chi(y_{\q}) \notin \q$ and $\chi(x_{\p})
\in \p$. Since $\chi(y) \equiv \chi(y_{\q}) \mod \q$ and $\chi(x) \equiv \chi(x_{\p}) \mod \p$, we have
$\chi(y) \notin \q$ and $\chi(x) \in \p$, which contradicts $P_{\p,x} \subseteq P_{\q,y}$.

If $\q = 0 $, it is well known that there is $\alpha_i \in R(S)$ such that 
$\alpha_i(y_i) \neq 0 $ and $\alpha_i(z) = 0$ if $z$ is not $S$-conjugate 
to $y_i$. Note that $b_i = \alpha_i(y_i)$ belongs to $A$. We consider 
\[ \alpha = \sum_{i=1}^n \left( \prod_{j \neq i} b_j \right) \alpha_i \]
which is a character in $R_A(S)$. It satisfies 
\[ \alpha(y_j) = \prod_{i=1}^n b_i \neq 0 \]
for all $j$ and $\alpha(z) = 0$ if $z$ is not $S$-conjugate to $y_i$ for any $i$. Since this character is constant
on each $\Ff$-conjugacy class, it belongs to $R_A(\Ff)$. And it satisfies $\alpha(x_{\p}) = 0$ and $\alpha(y_{\q}) \neq 0$.
Since $\alpha(x) \equiv \alpha(x_{\p}) \mod \p$ and $\alpha(y) \equiv \alpha(y_{\q}) \mod \q$, and $\p=\q=(0)$ in this
case, we obtain that $\alpha(x) \in \p$ and $\alpha(y) \notin \q$. This is again a contradiction with $P_{\p,x} \subseteq P_{\q,y}$.
\end{proof}

Note that all the prime ideals in $A$ are either $(0)$ or maximal ideals, hence
we can conclude:

\begin{prop}
\label{ClassificationIdeals}
The prime ideals of $R_A(\Ff)$ are all of the form $P_{\p,x}$. Two such
ideals $P_{\p,x} = P_{\q,y}$ coincide if and only if $\p=\q$ and $x_{\p} = y_{\q}$.
If $\p \neq (0)$, then $P_{\p,x}$ is a maximal prime ideal, while $P_{(0),x}$
is a minimal prime ideal. The maximal prime ideals containing $P_{(0),x}$ are
the ideals $P_{\p,x}$ with $\p \neq (0)$. The minimal prime ideals contained
in $P_{\p,x}$ for a certain $\p \neq (0)$ are the ideals $P_{(0),y}$ with $x_{\p} = y $.
\end{prop}

In short there are the following ideals. There is a $P_{(0),x}$ for each $x \in X_{\Ff}$, 
a $P_{\p,x}$ for each $x \in X_{\Ff}$ and each nontrivial prime ideal $\p$ of
$A$ with $q(\p) \neq p$, and a $P_{\p,1}$ for each nontrivial prime ideal $\p$ of $A$
with $q(\p)=p$.

\begin{cor}
The spectra of $R_A(\Ff)$ and $R(\Ff)$ are connected in the Zariski topo-logy.
\end{cor}

\begin{lem}
\label{AugmentationIdeal}
The ideal $P_{(0),1}$ coincides with $ I(\Ff) \otimes A  $.
\end{lem}

\begin{proof}
See the proof of Lemma 6.4 in \cite{A}.
%It is clear that $I(\Ff) \otimes A \subseteq P_{(0),1}$. On the
%other hand, any element of $R_A(\Ff)$ can be written as a sum of
%characters of the form $b_m \theta_m \chi_m$, where the $\chi_m \in R(\Ff)$, 
%the $\theta_m$ are roots of unity in $A$ which are linearly independent 
%over $\IZ$ and $b_m$ are integers. If this sum happens to be in 
%$P_{(0),1}$, this would imply $\sum b_m \theta_m n_m = 0$ for certain integers $n_m$ 
%(the virtual dimensions of $\chi_m$). This can only happen if all the $n_m=0$, that is, 
%if $\chi_m \in I(\Ff)$ for all $m$.
\end{proof}

Given $y \in X_S$ and a prime ideal $\p$ of $A$, we consider 
\[ Q_{\p,y} = \{ \chi \in R_A(S) \mid \chi(y) \in \p \} \]
and denote by $\res$ the maps $R_A(\Ff) \to R_A(S)$ and 
$R(\Ff) \to R(S)$ induced by restriction.

\begin{lem}
\label{Auxiliar}
If $\res^{-1}(Q_{\p,y}) = P_{\q,1}$, then $ Q_{\p,y} = Q_{\q,1}$.
\end{lem}

\begin{proof}
Note that $\res^{-1}(Q_{\p,y}) = P_{\p,r(y)}$, hence $\p = \q$
and $r(y)_{\p} = 1$. Hence $y_{\p}=1$ and so $Q_{\p,y} = Q_{\q,1}$.
\end{proof}

\begin{lem}
The prime ideals of $R(S)$ which contain $\res(I(\Ff))$ are the same
as those which contain $I(S)$.
\end{lem}

\begin{proof}
The proof is practically the same as that of Lemma 6.7 in \cite{A}. Since 
$\res(I(\Ff)) \subseteq I(S)$, the prime ideals that contain $I(S)$ also 
contain $\res(I(\Ff))$.

Let $K$ be a prime ideal of $R(S)$ which contains $\res(I(\Ff))$.
By Lemma \ref{PrimeIdeals}, there is a prime ideal $\p$ of $A$
and $x \in X_S$ such that $ K = R(S) \cap Q_{\p,x}$. Therefore
$\res(I(\Ff)) \subseteq Q_{\p,x}$. Recall that $ I(\Ff) \otimes A=P_{(0),1}$
by Lemma \ref{AugmentationIdeal} and so
\[ \res(P_{(0),1}) = \res(I(\Ff) \otimes A) \subseteq \res(I(\Ff)) \otimes A \subseteq Q_{\p,x} \]
and so $ P_{(0),1} \subseteq \res^{-1}(Q_{\p,x})$. By Lemma \ref{ClassificationIdeals},
we have $\res^{-1}(Q_{\p,x}) = P_{\q,1}$ for a certain prime ideal $\q$ of $A$ and
by Lemma \ref{Auxiliar}, $Q_{\p,x} = Q_{\q,1}$. In particular $I(S) \subseteq Q_{\p,x}$
and so $I(S) \subseteq K$.
\end{proof}

\begin{thm}
\label{TopologiesOnR(S)}
The $I(S)$-adic topology
of $R(S)$ coincides with its $I(\Ff)$-adic topo-logy.
\end{thm}

\begin{proof}
Just like Theorem 6.1 in \cite{A}, it follows from the previous lemma and 
the fact that in a Noetherian ring, the $J$-adic topology is the same as the $\rad(J)$-adic topology. 
%Here $\rad(J)$ is the intersection of all the prime ideals which contain $J$.
\end{proof}

\begin{cor}
\label{TopologiesCoincide}
Let $P$ be a subgroup of $S$. Then the $I(P)$-adic topology
of $R(P)$ coincides with its $I(\Ff)$-adic topology.
\end{cor}

\section{The completion theorem}
\label{section:completion}

In this section we show a completion theorem for the $K$-theory
of a $p$-local finite group $(S,\Ff,\Ll)$. Note that the homomorphism
\[ K(|\Ll| \pcom) \to K(BS) \]
induced by the map $BS \to |\Ll| \pcom$ coincides with composing the
isomorphism 
\[ K(|\Ll| \pcom) \cong \higherlim{\Or(\Ff^c)}{} K(BP) \]
from Theorem 4.2 of \cite{CCM} with the projection to the factor $K(BS)$.
In particular, it is injective and we identify $K(|\Ll| \pcom)$ as the subring
of stable elements of $K(BS)$. That is, elements $y$ such that $\res_P^S(y) = f^*(y)$ 
for all $P \leq S$ and any $f \in \Hom_{\Ff}(P,S)$. Here $\res_P^S$ is the map
induced in $K$-theory by the inclusion of $P$ in $S$.

\begin{thm}
\label{completiontheorem}
Given a $p$-local finite group $(S,\Ff,\Ll)$, the completion of $R(\Ff)$ with respect to the ideal $I(\Ff)$ is isomorphic 
to the $K$-theory ring of $|\Ll| \pcom$.
\end{thm}

\begin{proof}
The map $ \alpha_G \colon R(G) \to K(BG)$ constructed in Section 7 of \cite{A}
is natural with respect to group homomorphisms, hence it defines a chain map
\[ \alpha_* \colon C^*(\Or(\Ff^c);R) \to C^*(\Or(\Ff^c);K) \]
where $R$ and $K$ are the contravariant functors $\Or(\Ff^c) \to \Ab$ which send $P$ to $R(P)$ and
$K(BP)$, respectively. Here $C^*(D;F)$ denotes the cochain complex of the category $D$ with coefficients
in a contravariant functor $F \colon D \to \Ab$, which can be used to compute the higher limits of the
functor $F$ (see Lemma 2 in \cite{O0} for instance). 

For any morphism $f \colon P \to Q$ in $\Or(\Ff^c)$, we have a commutative diagram
\[
\diagram
R(\Ff) \dto_{\res_P^S \res} \drto^{\res_Q^S \res} & \\
R(P) \rto_{f^*} & R(Q)
\enddiagram
\]
and therefore $f^*$ is a homomorphism of $R(\Ff)$-modules. So we can regard
$R$ as a contravariant functor $\Or(\Ff^c) \to R(\Ff)-\Mod$ and $C^*(\Or(\Ff^c);R)$
as a complex of $R(\Ff)$-modules. Completion with respect to the ideal $I(\Ff)$
is an additive functor, so we have an isomorphism of complexes
\[ C^*(\Or(\Ff^c);R)^{\wedge}_{I_{\Ff}} \cong C^*(\Or(\Ff^c);R^{\wedge}_{I_{\Ff}}) \]
where $R^{\wedge}_{I_{\Ff}} \colon \Or(\Ff^c) \to \Ab$ is the contravariant functor
that sends $P$ to $R(P)^{\wedge}_{I_{\Ff}}$.

The homomorphism $\alpha_P$ induces an isomorphism $R(P) ^{\wedge}_{I(P)} \to K^*(BP)$ by
Atiyah's completion theorem and the $I(P)$-adic topology coincides with the $I(\Ff)$-adic 
topo-logy by Corollary \ref{TopologiesCoincide}. Therefore the chain map $\alpha_*$ induces 
an isomorphism of chain complexes.
\[ C^*(\Or(\Ff^c);R^{\wedge}_{I_{\Ff}}) \cong C^*(\Or(\Ff^c);K) \]

On the other hand, since $R(\Ff)$ is a Noetherian ring, completion of finitely ge-nerated 
$R(\Ff)$-modules with respect to the ideal $I(\Ff)$ is exact. Therefore we obtain an
isomorphism of abelian groups
\begin{align*}
K(|\Ll| \pcom) & \cong \higherlim{\Or(\Ff^c)}{} K(BP) \\
               & \cong H^0( C^*(\Or(\Ff^c);K) ) \\
               & \cong H^0( C^*(\Or(\Ff^c);R )^{\wedge}_{I_{\Ff}} ) \\
               & \cong H^0( C^*(\Or(\Ff^c);R ) )^{\wedge}_{I_{\Ff}}  \\
               & \cong \left( \higherlim{\Or(\Ff^c)}{} R(P) \right) ^{\wedge}_{I_{\Ff}}  \\
               & \cong R(\Ff) ^{\wedge}_{I_{\Ff}} 
\end{align*}
where the first isomorphism holds by Theorem 4.2 in \cite{CCM}. This isomorphism
is induced by restriction to $K(BS)$ and so the isomorphism from $R(\Ff) ^{\wedge}_{I_{\Ff}}$
to $K(|\Ll| \pcom)$ is induced by the restriction of the ring isomorphism $R(S)^{\wedge}_{I_{\Ff}} \to K(BS)$ 
to $R(\Ff) ^{\wedge}_{I_{\Ff}}$. This shows that in fact it is a ring isomorphism.
\end{proof}

For an abelian group $M$, we use the notation $M \pcom = M \otimes \IZ \pcom$.

\begin{prop}
Given a $p$-local finite group $(S,\Ff,\Ll)$, the completion of $R(\Ff) \pcom$ with respect to the ideal $I(\Ff) \pcom$ is isomorphic to the $p$-adic $K$-theory ring of $|\Ll| \pcom$.
\end{prop}

\begin{proof}
It is well known that the topology induced by $I(S)$ coincides with the topo-logy induced
by $pI(S)$ and so
\[ \widetilde{K}(BS) \cong I(S) ^{\wedge}_{I(S)} \cong I(S) ^{\wedge}_{pI(S)} \cong J(S) ^{\wedge}_{J(S)} \]
where $J(S)=I(S) \pcom$. Hence the induced map $R(S) \pcom \to K(BS ; \IZ \pcom)$ induces an 
isomorphism $(R(S) \pcom) ^{\wedge}_{J(S)} \to K(BS ; \IZ \pcom)$
since
\[ (R(S) \pcom) ^{\wedge}_{J(S)} \cong (\IZ \pcom \oplus J(S)) ^{\wedge}_{J(S)} \cong \IZ \pcom \oplus \widetilde{K}(BS) \cong \IZ \pcom \oplus \widetilde{K}(BS ; \IZ \pcom) = K(BS;\IZ \pcom) \]
The last isomorphism holds because $\widetilde{H}^k(BS;\IZ[1/p])=0$ and $BU$ is simply connected and so
\[ \widetilde{K}(BS) \cong [BS,BU] \cong [BS,BU \pcom] = \widetilde{K}(BS ; \IZ \pcom) \]
by Theorem 1.4 in \cite{Mi}. The same argument in the proof of the previous theorem shows that the completion of $R(\Ff) \pcom $ at the ideal $J(\Ff) = I(\Ff) \pcom$ is isomorphic to the $p$-adic $K$-theory ring $K(|\Ll| \pcom;\IZ \pcom)$.
\end{proof}

\section{Twistings and twisted representations}
\label{section:twistedrepresentations}

In this section we introduce twisted representations of 
fusion systems and relate them to representations of central
extensions. We begin by determining an appropriate model for 
the twistings classified by the third integral cohomology
group.

Let $|\Ll| \pcom$ be the classifying space of the $p$-local finite group $(S,\Ff,\Ll)$
and let $p^n$ be the order of $S$. By Theorem 4.2 in \cite{CCM}, there is an isomorphism
\[ H^3(|\Ll| \pcom ; \IZ) \cong \higherlim{\Or(\Ff)}{} H^3(BS;\IZ) \]
and therefore $ p^n H^3(|\Ll| \pcom ;\IZ) = 0$. The exact sequence $0 \to \IZ \stackrel{p^n}{\to} \IZ \to \IZ/p^n \to 0 $ of coefficients leads to a short exact sequence
\[ 0 \longrightarrow H^2(|\Ll| \pcom ;\IZ) \longrightarrow H^2(|\Ll| \pcom ;\IZ/p^n) \stackrel{\beta_n}{\longrightarrow} H^3(|\Ll| \pcom;\IZ) \longrightarrow 0 \]
and this shows that any twisting in $H^3(|\Ll| \pcom;\IZ)$ comes from an element in
$H^2(|\Ll| \pcom;\IZ/p^n)$.

By these considerations and Lemma 6.12 in \cite{BCGLO}, we can model the twistings of $|\Ll| \pcom$ 
by $2$-cocycles on the quasicentric linking system $\Ll^q$ (see Definition 1.9 in \cite{BCGLO}) with 
values in a cyclic $p$-group $A$. That is, $\alpha$ is a function from pairs of composable morphisms 
in the quasicentric linking system $\Ll^q$ to $A$ such that $\alpha(f,g)$ vanishes if $f$ or $g$ is 
an identity morphism and for any triple $f$, $g$, $h$ of composable morphisms, the cocycle condition 
is satisfied:
\[ \alpha(g,h) - \alpha(gf,h) + \alpha(f,hg) - \alpha(f,g) = 0 \]
The distinguished morphism $ \delta_S \colon S \to \Aut_{\Ll^q}(S)$ together
with the inclusion of $\Aut_{\Ll^q}(S)$ in $\Ll^q$ define a functor $\calB S \to \Ll^q$.
Here $\calB S$ denotes the category with one object, whose set of morphisms is $S$ and
composition is given by $s \circ t = s \cdot t $, the multiplication in $S$. Restriction 
along this functor defines a $2$-cocycle for $\calB S$ with values in $A$ which we also 
denote by $\alpha$. Note that $\alpha$ is not a $2$-cocycle for $S$ in the usual sense,
but if we define $\alpha'(s,t) = \alpha(t,s)$, it is easy to check that $\alpha'$ satisfies
\[ \alpha'(t,u) - \alpha'(st,u) + \alpha'(s,tu) - \alpha'(s,t) = 0 \]
By abuse of notation, we denote this $2$-cocycle by $\alpha$ as well and refer to it
as a $2$-cocycle for $(S,\Ff,\Ll)$.

Now that we have determined a model for our twistings, we recall how we use $2$-cocycles 
to define twisted representations of groups, following \cite{K94}. From now on, $A$ is a cyclic
$p$-group which we regard inside $S^1$ as the subgroup of corresponding roots of unity.

\begin{defn}
Let $\alpha$ be a $2$-cocycle for a group $S$ with values in $A$. An $\alpha$-twisted 
representation of $S$ is a map $\rho \colon S \to U_n$ such that $\rho(1)$ is the identity matrix and
\[ \rho(s) \rho(t)= \alpha(s, t) \rho(st). \]
for any $s$, $t \in S$.
\end{defn}  

Two $\alpha$-twisted representations $\rho_1$, $\rho_2 \colon S \to U_n$ are equivalent if
there is a linear isomorphism $f \colon \IC^n \to \IC^n$ such that $ \rho_1 = f \rho_2 f^{-1}$.
We denote by $^{\alpha} \Rep_n(S)$ the set of equivalence classes of $\alpha$-twisted $n$-dimensional
representations of $S$. 

\begin{rem}
Recall that an $\Ff$-invariant representation of $S$ is a representation $\rho$
such that $\rho_{|P} = \rho_{|f(P)} \circ f$ in $\Rep(P,U_n)$ for any $f \in \Hom_{\Ff}(P,S)$.
However, note that $j_P$, $f \in \Hom_{\Ff}(P,S)$ induce maps
\begin{gather*}
f^* \colon {^{\alpha} \Rep_n(S) } \to {^{f^* \alpha} \Rep_n(P) } \\
j_P^* \colon {^{\alpha} \Rep_n(S) } \to {^{j_P^* \alpha} \Rep_n(P) }
\end{gather*}
But even if $[\alpha] \in H^2(S;A)$ is invariant with respect to the fusion
system, it is not necessarily the case that $ f^* \alpha = j_P^* \alpha$. In
general, they differ by a coboundary and so there is a bijection between 
the two sets of twisted representations of $P$, but they may not be equal.
\end{rem}

Recall that a $2$-cocycle $\alpha \colon S \times S \to A$ determines a central extension 
of $A$ by $S$ as follows. Let $S_{\alpha}$ be the set $A \times S$ 
with the operation
\[ (a_1,s_1) (a_2,s_2) = (a_1 + a_2 + \alpha(s_1,s_2), s_1 s_2) \]
This gives $S_{\alpha}$ the structure of a group which fits into a central extension
\[ 1 \to A \to S_{\alpha} \to S \to 1 \]
with the obvious homomorphisms. An $\alpha$-twisted representation $ \rho \colon S \to U_n$ 
defines a representation $\rho_{\alpha} \colon S_{\alpha} \to U_n$ by
\[ \rho_{\alpha}(a,s) = a \rho(s) \]
where we see $a$ as a complex number. We refer to $\rho_{\alpha}$ as the untwisting
of $\rho$.

By Theorem 6.13 in \cite{BCGLO} there is a bijective correspondence between the set of equivalence classes 
of central extensions of $(S,\Ff,\Ll)$ by $A$ and elements of $H^2(|\Ll| \pcom;A)$. Given a class in 
$H^2(|\Ll| \pcom;A)$ and a $2$-cocycle $\alpha$ for the corresponding class in $H^2(S;A)$, we denote by 
$(S_{\alpha},\Ff_{\alpha},\Ll_{\alpha})$ the corresponding $p$-local finite group.

\begin{defn}
\label{DefTwistedRepresentation}
Let $\alpha$ be a $2$-cocycle for the $p$-local finite group $(S,\Ff,\Ll)$. We say that an $\alpha$-twisted 
representation $\rho$ of $S$ is $\Ff$-invariant if $\rho_{\alpha}$ is $\Ff_{\alpha}$-invariant.
\end{defn}

Let us denote by $^{\alpha} \Rep(\Ff)$ the set of isomorphism classes of $\alpha$-twisted
representations of $S$ that are $\Ff$-invariant. This set is a monoid under direct sum and
the following is a natural definition.

\begin{defn}
\label{DefTwistedRepresentationRing}
Let  $\alpha$  be  a  $2$-cocycle for the $p$-local finite group $(S,\Ff,\Ll)$.  
The $\alpha$-twisted representation group of $\Ff$ is the Grothendieck group of 
the monoid $^{\alpha} \Rep(\Ff)$. We will denote it by $^{\alpha}R(\Ff)$. 
\end{defn}

The abelian group $^{\alpha}R(\Ff)$ is an $R(\Ff)$-module with the action
\begin{gather*}
R(\Ff) \times {^{\alpha}R(\Ff)} \to {^{\alpha}R(\Ff)} \\
(V,W) \quad \mapsto \quad V \otimes W 
\end{gather*}

Note that representations obtained by untwisting are $A$-representations
in the following sense.

\begin{defn}
We will say that a representation of $S_{\alpha}$ is an $A$-representation if 
$A$ acts by complex multiplication.
\end{defn}

Let $\Rep(S_{\alpha};A)$ be the monoid of isomorphism classes of $A$-representations of $S_{\alpha}$ 
and $R(S_{\alpha};A)$ the corresponding Grothendieck group. By Lemma 1.2 in Chapter 4 of \cite{K94}, 
the process of untwisting determines isomorphisms
\begin{gather*}
^{\alpha} \Rep(S) \cong \Rep(S_{\alpha};A) \\
^{\alpha} R(S) \cong R(S_{\alpha};A) 
\end{gather*}
Let us denote by $\Rep(\Ff_{\alpha};A)$ the monoid of isomorphism classes of $A$-representations of $S_{\alpha}$
that are $\Ff_{\alpha}$-invariant, that is
\[ \Rep(\Ff_{\alpha};A) = \Rep(S_{\alpha};A) \cap \Rep(\Ff_{\alpha}) \]
and let $R(\Ff_{\alpha};A)$ be the corresponding Grothendieck group. This is an $R(\Ff)$-module
with the action
\begin{gather*}
R(\Ff) \times R(\Ff_{\alpha};A) \to R(\Ff_{\alpha};A) \\
(V,W) \quad \mapsto \quad V \otimes W 
\end{gather*}
where $(a,s) v \otimes w = sv \otimes (a,s) w $.

\begin{lem}
\label{UntwistingRepresentations}
Let $\alpha$ be a $2$-cocycle for the $p$-local finite group $(S,\Ff,\Ll)$ and let 
$(S_{\alpha},\Ff_{\alpha},\Ll_{\alpha})$ be the corresponding central extension. 
Then untwisting of $\alpha$-twisted representations of $S$ induces an isomorphism 
of $R(\Ff)$-modules
\[ {^{\alpha} R(\Ff)} \to R(\Ff_{\alpha};A) \]
\end{lem}

\begin{proof}
Consider the restriction of the untwisting isomorphism $^{\alpha} R(S) \cong R(S_{\alpha};A)$
to ${^{\alpha} R(\Ff)}$. By definition, $\rho$ is an $\Ff$-invariant $\alpha$-twisted representation
if and only if $\rho_{\alpha}$ is $\Ff_{\alpha}$-invariant, hence the restriction ${^{\alpha} R(\Ff)} \to R(\Ff_{\alpha};A)$
is an isomorphism of abelian groups. It is straightforward to check that it is an isomorphism of $R(\Ff)$-modules.
\end{proof}

\begin{cor}
\label{TwistedRepFG}
Let $\alpha$  be a  $2$-cocycle for the $p$-local finite group $(S,\Ff,\Ll)$. Then 
$^{\alpha}R(\Ff)$ is a finitely generated module over the representation ring $R(\Ff)$. 
\end{cor}

\begin{proof}
It follows from the previous lemma and the fact that $R(\Ff_{\alpha};A)$ is finitely generated as
an abelian group. 
\end{proof}

For the convenience of the reader, we include some terminology about
bisets from \cite{BLO}. Recall that an $(S,S)$-biset is a set with a left and a right action
of $S$ which commute with each other. If $P \leq S$ and $\varphi \in \Hom(P,S)$,
then $S \times_{P,\varphi} S$ denotes the $(S,S)$-biset $(S \times S)/{\sim}$
where $(x,gy) \sim (x\varphi(g),y)$ for $x$, $y \in S$ and $g \in P$, with the
obvious left and right action of $S$. If $B$ is an $(S,S)$-biset, then for
$P \leq S$ and $\varphi \in \Hom(P,S)$, we denote by $B_{(P,S)}$ the restriction
of $B$ to a $(P,S)$-biset, and let $B_{(\varphi,S)}$ denote the $(P,S)$-biset
where the left $P$-action is induced by $\varphi$. Given a saturated fusion 
system $\Ff$ over $S$, we say that an $(S,S)$-biset is a characteristic biset
for $\Ff$ if it satisfies the three conditions in Proposition 5.5 of \cite{BLO}.

The following result is mentioned in the proof of Lemma 5.6 (d)
of \cite{BLO} without proof. We include a proof here for completeness.

\begin{lem}
\label{CentralBiset}
Let $\Ff_{\alpha}$ be a saturated fusion system over $S_{\alpha}$ and let $A$ be a subgroup
of $S_{\alpha}$ which is central in $\Ff_{\alpha}$. Then there exists a characteristic biset
$\Lambda$ for $\Ff_{\alpha}$ such that $ax=xa$ for all $x \in \Lambda$ and all $a \in A$.
\end{lem}

\begin{proof}
We know there exists a characteristic biset $\Omega$ for $\Ff_{\alpha}$ by Proposition 5.5 of \cite{BLO}.
Consider the conjugation action of $A$ on $\Omega$. Let us see that $\Lambda = \Omega^A$
is a characteristic biset for $\Ff_{\alpha}$ with the desired property. First of all, $\Lambda$ inherits 
left and right actions of $S_{\alpha}$ from $\Omega$ because $A$ is central in $S_{\alpha}$. And if 
$x \in \Lambda$ and $a \in A$, then
\[ a x = a a^{-1} x a = x a \]
so it satisfies the desired property. It remains to show that it satisfies the three conditions
of a characteristic biset from Proposition 5.5 of \cite{BLO}. \newline

\noindent (1) The biset $\Lambda$ is a disjoint union of bisets of the form $S_{\alpha} \times_{(P,\varphi)} S_{\alpha}$
with $P \leq S_{\alpha}$ and $\varphi \colon P \to S_{\alpha}$ in $\Ff_{\alpha}$. \newline

If $[x,y]$ belongs to one the components $S_{\alpha} \times_{(P,\varphi)} S_{\alpha}$
in the decomposition of $\Omega$ and $A$ is contained in $P$, then
\[ a [x,y] a^{-1} = [ax,ya^{-1}] = [xa,a^{-1}y] = [xa\varphi(a)^{-1},y] = [xaa^{-1},y] = [x,y] \]
and therefore $S_{\alpha} \times_{(P,\varphi)} S_{\alpha}$ is contained in $\Lambda$. If $A$ is not contained in $P$, 
let $ [x,y] \in (S_{\alpha} \times_{(P,\varphi)} S_{\alpha})^A $ and $a \in A - P$. Then
\[ [ax,ya^{-1}] = [x,y] \]
from where $a \in P$. This is a contradiction and so $(S_{\alpha} \times_{(P,\varphi)} S_{\alpha})^A $ is empty. We conclude
that the biset $\Lambda$ is the disjoint union of the bisets of the form $S_{\alpha} \times_{(P,\varphi)} S_{\alpha}$
that appear in the decomposition of $\Omega$ and that satisfy $ A \subseteq P$. \newline

\noindent (2) For each $P \leq S_{\alpha}$ and each $\varphi \in \Hom_{\Ff_{\alpha}}(P,S_{\alpha})$, the $(P,S_{\alpha})$-bisets $\Lambda_{(P,S_{\alpha})}$
and $\Lambda_{(\varphi,S_{\alpha})}$ are isomorphic. \newline 

By Alperin's fusion lemma, it suffices to prove that $\Lambda_{(Q,S_{\alpha})}$ and $\Lambda_{(\psi,S_{\alpha})}$
are isomorphic $(Q,S_{\alpha})$-bisets for any $\Ff_{\alpha}$-centric subgroup $Q$ and any $\psi \colon Q \to S_{\alpha}$
in $\Ff_{\alpha}$. Since $\Omega_{(Q,S_{\alpha})}$ and $\Omega_{(\psi,S_{\alpha})}$ are isomorphic $(Q,S_{\alpha})$-bisets, there exists a $(Q,S_{\alpha})$-equivariant
bijection $ \theta \colon \Omega_{(Q,S_{\alpha})} \to \Omega_{(\psi,S_{\alpha})}$. Given $x \in \Lambda$ and $a \in A$, we
have
\[ a \theta(x) a^{-1} = \theta(\psi(a)xa^{-1}) = \theta(axa^{-1}) = \theta(x) \]
where the first equality holds because $Q$ is $\Ff_{\alpha}$-centric and so it contains $A$. Hence $\theta$
restricts to a $(Q,S_{\alpha})$-equivariant bijection $\Lambda_{(Q,S_{\alpha})} \to \Lambda_{(\psi,S_{\alpha})}$. \newline

\noindent (3) The number $|\Lambda| / |S_{\alpha}| $ is congruent to $1$ mod $p$. \newline

By the first part, we know we can obtain $\Lambda$ from $\Omega$ by removing the summands $S_{\alpha} \times_{(P,\varphi)} S_{\alpha}$ where
$P$ does not contain $A$. 
\[ |\Lambda| / |S_{\alpha}| = |\Lambda/S_{\alpha}| = |\Omega/S_{\alpha}| - \sum_{A \nleq P} |S_{\alpha} \times_{(P,\varphi)} S_{\alpha}| \]
Note that $|S_{\alpha} \times_{(P,\varphi)} S_{\alpha}| = |S_{\alpha}/P|$ and all the summands that
we are removing satisfy $ P \neq S_{\alpha}$ since $S_{\alpha}$ contains $A$. Therefore $|\Lambda| / |S_{\alpha}| 
\equiv |\Omega|/|S_{\alpha}| \mod p$, thus congruent to $1$ mod $p$.
\end{proof}

\begin{thm}
\label{InductionBReps}
Let $\rho$ be an $A$-representation of $S_{\alpha}$. Then there exists an $A$-representation of $S_{\alpha}$ which
is $\Ff_{\alpha}$-invariant and that contains $\rho$ as a direct summand.
\end{thm}

\begin{proof}
Since $A$ is central in $\Ff_{\alpha}$, by Lemma \ref{CentralBiset} there exists
a characteristic biset $\Lambda$ for $\Ff_{\alpha}$ such that $ax=xa$ for all $a \in S$
and all $x \in \Lambda$. Following Proposition 3.8 in \cite{CCM}, if $\rho$ is an $n$-dimensional 
representation, we consider the action of $S_{\alpha}$ on $ \IC[\Lambda] \otimes_{S_{\alpha}} \IC^n$ 
induced by the left action of $S_{\alpha}$ on $\Lambda$. Since $\Lambda$ is characteristic, the 
same argument goes through to show that this representation is $\Ff_{\alpha}$-invariant and contains $\rho$ as a direct
summand. Moreover, $A$ acts by complex multiplication on this new representation since
\[ a (x \otimes v) = ax \otimes v = xa \otimes v = x \otimes \rho(a)(v) \]
and $\rho$ is an $A$-representation.
\end{proof}

We now prove an analogue of Proposition 5.7 in \cite{CCM} for $A$-representations.

\begin{prop}
\label{StableElementsAReps}
There is an isomorphism $ R(\Ff_{\alpha};A) \cong \higherlim{\Or(\Ff_{\alpha}^c)}{} R(Q;A) $
of modules over $R(\Ff)$.
\end{prop}

\begin{proof}
The obvious map $ R(\Ff_{\alpha};A) \to \higherlim{\Or(\Ff_{\alpha}^c)}{} R(Q;A) $ is clearly 
a monomorphism of $R(\Ff)$-modules. To see that it is surjective it suffices to show that given 
an element $([\rho_Q] - [\beta_Q])_Q$ in the limit, there exists $[\rho] - [\beta]$ in $R(\Ff_{\alpha};A)$ such that 
\[ [\rho] - [\beta] = [\rho_S] - [\beta_S] \]
By Lemma \ref{InductionBReps}, there is a $\Ff_{\alpha}$-invariant $A$-representation $\rho_S^{\Ff_{\alpha}}$ such that
\[ \rho_S^{\Ff_{\alpha}} = \rho_S \oplus \rho' \]
Therefore
\[ [\rho_S] - [\beta_S] = [\rho_S^{\Ff_{\alpha}}] - [\beta_S \oplus \rho'] \]
and note that $\beta_S \oplus \rho'$ is $\Ff_{\alpha}$-invariant because $\rho_S^{\Ff_{\alpha}}$ and the
difference are both $\Ff_{\alpha}$-invariant. 
\end{proof}

We end this section with the following result, which will be useful in
the proof of Theorem \ref{TwistedCompletion}.

\begin{lem}
\label{SameTopologies}
Let $P$ be an $\Ff_{\alpha}$-centric subgroup of $S_{\alpha}$. The $I(\Ff)$-adic and $I(P/A)$-adic 
topologies coincide for $R(P;A)$.
\end{lem}

\begin{proof}
The action of $R(\Ff)$ on $R(P;A)$ factors through the
restriction from $R(\Ff)$ to $R(P/A)$. By Corollary \ref{TopologiesCoincide},
the $I(P/A)$-adic and $I(\Ff)$-adic topologies coincide for $R(P/A)$
and therefore for $R(P;A)$.
\end{proof}

\section{The twisted completion theorem}
\label{section:twistedcompletion}

In this section we prove the general completion theorem for $p$-local finite groups. 
First we determine that the twisted $K$-theory of the classifying space can be computed
by stable elements in a certain sense and use Lahtinen's completion
maps in twisted $K$-theory \cite{L} to induce a completion map for the group 
of twisted representations of a fusion system. Since this group is also computed
by stable elements, we can use the same argument as in Theorem \ref{completiontheorem}.

Throughout this section we let $A$ be a cyclic $p$-group of order $p^n$ and let
$\beta_n \colon H^2(B;A) \to H^3(B;\IZ)$ be the connecting morphism induced by 
the short exact sequence of coefficients $ 0 \to \IZ \stackrel{p^n}{\to} \IZ \to A \to 0 $ as
in Section \ref{section:twistedrepresentations}.

We begin by giving an alternative description of twisted $K$-theory for twistings in the
image of $\beta_n$. The inclusion of $A$ in $S^1$ induces a homomorphism
$BA = U(\Hi)/A \to PU(\Hi)$ and therefore $BA$ acts on $\Fred'(\Hi)$. Let $ \pi \colon E \to B$
be a principal $BA$-bundle. We define $ K(E;\pi)$ to be the set of equivariant homotopy classes of $BA$-equivariant maps
$ \phi \colon E \to \Fred'(\Hi)$, that is, $ \phi ( e \cdot g) = g^{-1} \phi(e) $ for all $g \in BA$.

\begin{lem}
\label{UntwistingKTheory}
Let $ \pi \colon E \to B $ be a principal $BA$-bundle classified by $\alpha \in H^2(B;A)$. 
Then there is a bijection
\[ K(E;\pi) \cong {}^{\beta_n(\alpha)} K(B)  \]
\end{lem}

\begin{proof}
Let $P \to B$ be the projective bundle classified by $\beta_n(\alpha)$. Recall that ${}^{\beta_n(\alpha)} K(B)$ is 
the set of homotopy classes of sections of the associated bundle $ P \times_{PU(\Hi)} \Fred'(\Hi) \to B$.
It is straightforward to check that Theorem 8.1 in Chapter 4 of \cite{Hu} is still valid for homotopy classes
of maps, hence this set is in bijective correspondence with the set of equivariant homotopy classes of 
$PU(\Hi)$-equivariant maps $ \phi \colon P \to \Fred'(\Hi)$, in the sense that $\phi(eg)=g^{-1}\phi(e)$.

Since $\beta_n(\alpha)$ is the composition of the classifying map $ \alpha \colon B \to BBA$ for $\pi$ with the map induced
by the inclusion of $A$ in $S^1$, the principal $PU(\Hi)$-bundle $ P \to B$ is isomorphic to 
$ E \times_{BA} {PU(\Hi)} \to B$. Therefore the set of equivariant homotopy classes of $PU(\Hi)$-equivariant
maps $ P \to \Fred'(\Hi)$ is in bijection with the set of equivariant homotopy classes of $BA$-equivariant
maps $ E \to \Fred'(\Hi)$.
\end{proof}

The group structure of ${}^{\alpha} K(B)$ is given by fiberwise composition in $\Fred'(\Hi)$. The
corresponding operation on $K(E;\pi)$ is therefore also induced by composition in $\Fred'(\Hi)$ and
the bijection in this lemma becomes an isomorphism of groups with this structure.

\begin{cor}
\label{UntwistingKTheoryPLocal}
Let $\alpha$ be a $2$-cocycle for the $p$-local finite group $(S,\Ff,\Ll)$ and let
$ \pi \colon |\Ll_{\alpha}| \pcom \to |\Ll| \pcom $ be the corresponding principal
$BA$-bundle. There is an isomorphism 
\[ K(|\Ll_{\alpha}| \pcom;\pi) \cong {^{\beta_n(\alpha)} K(|\Ll| \pcom)} \]
\end{cor}

\begin{rem}
Note that a commutative diagram
\[
\diagram
E' \rto^q \dto_{\pi'} & E \dto^{\pi} \\
B' \rto & B 
\enddiagram
\]
where $\pi'$ and $\pi$ are principal $BA$-bundles and $q$ is $BA$-equivariant, induces a homomorphism
$ q^* \colon K(E;\pi) \to K(E';\pi')$. This assignment is clearly functorial in the category of such maps. More generally,
given principal $BA$-bundles $\pi' \colon E' \to B' $ and $\pi \colon E \to B$ that fit into a commutative
diagram of spectra
\[
\diagram
\Sigma^{\infty} E' \rto^q \dto_{\Sigma^{\infty} \pi'} & \Sigma^{\infty} E \dto^{\Sigma^{\infty} \pi} \\
\Sigma^{\infty} B' \rto   & \Sigma^{\infty} B
\enddiagram
\]
where $q$ is $BA$-equivariant, we have an induced map $ q^* \colon K(E;\pi) \to K(E';\pi')$ because equivariant
homotopy classes of $BA$-equivariant maps $E \to \Fred'(\Hi)$ are in bijective correspondence with equivariant
homotopy classes of $BA$-equivariant stable maps $\Sigma^{\infty} E \to \Sigma^{\infty} \Fred'(\Hi)$. This assignment
is also functorial.
\end{rem}

\begin{rem}
By Proposition 4.9 of \cite{Rag}, the characteristic biset $\Lambda$ for $\Ff_{\alpha}$ from Lemma 
\ref{CentralBiset} determines a characteristic idempotent in $A(S_{\alpha},S_{\alpha}) \pcom$
and the stable summand determined by the corresponding idempotent $\omega$ in 
$\{ \Sigma^{\infty} BS_{\alpha}, \Sigma^{\infty} BS_{\alpha} \}$ is $\Sigma^{\infty} |\Ll_{\alpha}| \pcom$.

Note that the $(S_{\alpha},S_{\alpha})$-biset $\Lambda$ constructed in the proof of Lemma \ref{CentralBiset} is a disjoint union of $(S_{\alpha},S_{\alpha})$-bisets of the form $S_{\alpha} \times_{(P,\varphi)} S_{\alpha}$ where $P$ contains $A$. Let $S = S_{\alpha}/A$.
We can construct an $(S,S)$-biset $\Lambda/A$ by including a copy of $S \times_{(P/A,\varphi/A)} S$ for each copy
of $S_{\alpha} \times_{(P,\varphi)} S_{\alpha}$ in $\Lambda$. Now each $(S_{\alpha},S_{\alpha})$-biset $S_{\alpha} \times_{(P,\varphi)} S_{\alpha}$ in $\Lambda$ determines a commutative diagram of stable maps
\[ \xymatrixcolsep{5pc}
\diagram
\Sigma^{\infty} BS_{\alpha} \rto^{\tr_P^{S_{\alpha}}} \dto_{\Sigma^{\infty} Bq} & \Sigma^{\infty} BP \rto^{\Sigma^{\infty} B\varphi} \dto^{\Sigma^{\infty} Bq_{|BP}} & \Sigma^{\infty} BS_{\alpha} \dto^{\Sigma^{\infty} Bq} \\ 
\Sigma^{\infty} BS \rto_{\tr_{P/A}^S} & \Sigma^{\infty} B(P/A) \rto_{\Sigma^{\infty} B(\varphi/A) } & \Sigma^{\infty} BS
\enddiagram
\]
where $q \colon S_{\alpha} \to S$ is the quotient homomorphism and $\tr$ denotes the stable transfer map. These two
bisets determine stable maps $\omega$ and $\omega/A$, given by the sum of these stable maps above, each summand 
corresponding to a component of the biset. By the commutativity of the diagrams above, these stable maps fit into a
commutative diagram
\[
\diagram
\Sigma^{\infty} BS_{\alpha} \rto^{\omega} \dto_{\Sigma^{\infty} Bq} & \Sigma^{\infty} BS_{\alpha} \dto^{\Sigma^{\infty} Bq} \\ 
\Sigma^{\infty} BS \rto_{\omega/A} & \Sigma^{\infty} BS
\enddiagram
\]
\end{rem}
 
\begin{thm}
\label{StableElementsKATheory}
Let $\alpha$ be a $2$-cocycle for the $p$-local finite group $(S,\Ff,\Ll)$ and 
$ \pi \colon |\Ll_{\alpha}| \pcom \to |\Ll| \pcom $ the corresponding principal
$BA$-bundle. There is an isomorphism of abelian groups
\[ K(|\Ll_{\alpha}| \pcom;\pi) \cong \higherlim{\Or(\Ff_{\alpha}^c)}{} K(BQ;\pi_Q) \]
where $\pi_Q \colon BQ \to B(Q/A)$ is the pullback of $\pi$ under the composition of
the map induced by the inclusion of $Q/A$ in $S$ and the standard map $BS \to |\Ll| \pcom$.
\end{thm}

\begin{proof}
By the two previous remarks, one just needs to follow the proof of Theorem 4.2 of \cite{CCM} 
with the stable idempotent $\omega$ described above.
\end{proof}

\begin{thm}
\label{TwistedCompletion}
Given a $2$-cocycle $\alpha$ for the $p$-local finite group $(S,\Ff,\Ll)$,
the completion of $^{\alpha} R(\Ff)$ with respect to $I(\Ff)$ is isomorphic to 
$^{\beta_n(\alpha)} K(|\Ll| \pcom)$ as a module over $K(|\Ll| \pcom)$.
\end{thm}
   
\begin{proof}
For any subgroup $Q$ of $S$, the constant map $EQ \to *$ induces a homomorphism
\[ {^{\alpha}R(Q)} \to {^{\beta_n(\alpha)}K_Q(EQ)} \]
and ${^{\beta_n(\alpha)}K_Q(EQ)} \cong {^{\beta_n(\alpha)} K(BQ)}$ by Proposition 3.2 in \cite{FHT}.
We obtain a homomorphism ${^{\alpha}R(Q)} \to {^{\beta_n(\alpha)}K(BQ)}$. Now let $P$ be an 
$\Ff_{\alpha}$-centric subgroup of $S_{\alpha}$. Consider the composition
\[ R(P;A) \cong {^{\alpha}R(P/A)} \to {^{\beta_n(\alpha)} K(B(P/A))} \cong K(BP;\pi_P) \]
where the first isomorphism is given by untwisting, the second map is the map mentioned
in the previous paragraph and the last isomorphism comes from Lemma \ref{UntwistingKTheory}.
The composition sends a representation of $P$ to the homotopy class of the induced map
$BP \to \IZ \times BU$, hence it determines a natural transformation $\beta$ of functors 
$\Or(\Ff_{\alpha}^c) \to \Ab$.

Let $R(\mbox{ };A)$ and $R(\mbox{ };A)^{\wedge}_{I_{\Ff}}$ be the contravariant functors 
$\Or(\Ff_{\alpha}^c) \to \Ab$ which send $P$ to $R(P;A)$ and $P$ to $R(P;A)^{\wedge}_{I_{\Ff}}$,
respectively. We now follow the same argument from Theorem \ref{completiontheorem}. Completion 
of $R(\Ff)$-modules with respect to the ideal $I(\Ff)$ is an additive functor, so we have 
an isomorphism of complexes
\[ C^*(\Or (\Ff_{\alpha}^c );R(\mbox{ };A) )^{\wedge}_{I_{\Ff}} \cong C^*(\Or(\Ff_{\alpha}^c);R(\mbox{ };A)^{\wedge}_{I_{\Ff}}) \]

The map $\beta_P$ induces an isomorphism $R(P;A) ^{\wedge}_{I(P/A)} \to K^*(BP;\pi_P)$
of abelian groups by Theorem 1 in \cite{L}. Since the $I(P/A)$-adic topology coincides with the $I(\Ff)$-adic topology by Lemma \ref{SameTopologies},
the chain map $\beta_*$ induces an isomorphism of chain complexes
\[ C^*(\Or(\Ff_{\alpha}^c);R(\mbox{ };A)^{\wedge}_{I_{\Ff}}) \cong C^*(\Or(\Ff_{\alpha}^c);K(\mbox{ };\pi)) \]
where $K(\mbox{ };\pi)$ denotes the contravariant functor $\Or(\Ff_{\alpha}^c) \to \Ab$ which sends $P$
to $K(BP;\pi_P)$. Then we have the following isomorphisms
\begin{align*}
^{\beta_n(\alpha)} K(|\Ll| \pcom) & \cong K(|\Ll_{\alpha}| \pcom;\pi) \\
                         & \cong \higherlim{\Or(\Ff_{\alpha}^c)}{} K(BP;\pi_P) \\
               & \cong H^0( C^*(\Or(\Ff_{\alpha}^c);K(\mbox{ };\pi ) )) \\
               & \cong H^0( C^*(\Or(\Ff_{\alpha}^c);R(\mbox{ };A) )^{\wedge}_{I_{\Ff}} )
\end{align*}
where the first isomorphism holds by Corollary \ref{UntwistingKTheoryPLocal} and the second
isomorphism by Theorem \ref{StableElementsKATheory}. Since $R(\Ff)$ is a Noetherian ring and 
the $R(\Ff)$-modules $R(P;A)$ are finitely generated by Lemma \ref{TwistedRepFG}, completion 
of finitely generated $R(\Ff)$-modules with respect to the ideal $I(\Ff)$ is an exact functor and so
\begin{align*} 
H^0( C^*(\Or(\Ff_{\alpha}^c);R(\mbox{ };A) )^{\wedge}_{I_{\Ff}} ) & \cong H^0( C^*(\Or(\Ff_{\alpha}^c);R(\mbox{ };A) ) )^{\wedge}_{I_{\Ff}} \\
               & \cong \left( \higherlim{\Or(\Ff_{\alpha}^c)}{} R(P;A) \right) ^{\wedge}_{I_{\Ff}}  \\
               & \cong R_A(\Ff_{\alpha}) ^{\wedge}_{I_{\Ff}} \\
               & \cong {^{\alpha} R(\Ff)} ^{\wedge}_{I_{\Ff}}
\end{align*}
where the third isomorphism follows from Proposition \ref{StableElementsAReps} and the last isomorphism
from Lemma \ref{UntwistingRepresentations}. We obtain an isomorphism of abelian groups
\[ ^{\beta_n(\alpha)} K(|\Ll| \pcom) \cong {^{\alpha} R(\Ff)} ^{\wedge}_{I_{\Ff}} \]
which is induced by the restriction of the isomorphism of $K(BS)$-modules $^{\alpha} R(S)^{\wedge}_{I_{\Ff}} \to {^{\beta_n(\alpha)} K(BS)}$. Hence it is an isomorphism of $K(|\Ll| \pcom)$-modules.
\end{proof}  

\section{Computations}
\label{section:final}

In this final section, we include several computations of twisted and untwisted $K$-theory of classifying
spaces of $p$-local finite groups.

\begin{example}
Let $p \neq 2$ and let $\Ff$ be the saturated fusion system of $\Sigma_p$ over the
$p$-Sylow $S$ generated by $\sigma = (1,2,\ldots,p)$. The irreducible representations 
of $S \cong \IZ/p$ are the tensor powers of the one-dimensional representation $\rho$ 
that sends $\sigma$ to $ e^{2\pi i/p}$. Its representation ring is given by
\[ R(S) = \IZ[\rho]/(\rho^p-1) \]
The representation ring of $\Ff$ is generated by $1$ and
$ x = \rho + \rho^2 + \ldots + \rho^{p-1}$ since all the nontrivial
powers of $\sigma$ are conjugate in $\Sigma_p$. It satisfies $ x^2 = (p-1)1 + (p-2)x$, 
hence by Theorem \ref{completiontheorem},
\[ K( ( B\Sigma_p) ^{\wedge} _p) \cong \Big{[} \IZ[x]/(x^2-(p-2)x-p+1) \Big{]}^{\wedge}_{(x-p+1)} \cong \IZ[[y]]/(y^2+py) \]
where $y$ is the element in $K$-theory coming from $x-p+1$. Similarly, its $p$-adic $K$-theory
ring is given by $ K( ( B\Sigma_p) ^{\wedge} _p;\IZ \pcom) \cong  \IZ \pcom [[y]]/(y^2+py) $.
\end{example}

\begin{example}
Let $\Ff$ be the fusion system of $A_4$ over its $2$-Sylow 
\[ S=\{ 1,(1,2)(3,4),(1,3)(2,4), (1,4)(2,3) \} \]
The connecting map $ H^2((BA_4) ^{\wedge}_2;\IZ/2) \to H^3((BA_4) ^{\wedge}_2;\IZ)$ is an isomorphism
since $H^2((BA_4) ^{\wedge}_2;\IZ)=0$. Hence we identify both groups. The nontrivial element $\alpha$ in $ H^2((BA_4) ^{\wedge}_2;\IZ/2) \cong H^2(BA_4;\IZ/2) \cong \IZ/2$ corresponds to the fusion system $\Ff_{\alpha}$
of the central extension $SL_2(\IF_3)$ over a $2$-Sylow $S_{\alpha}$, which is isomorphic to $Q_8$. The only 
irreducible representation $\rho$ of $Q_8$ where the center acts by complex multiplication is the $2$-dimensional representation $\rho$ coming from the action of $Q_8$ on the quaternions, which satisfies
\[ \chi_{\rho}(g) = \left\{ \begin{array}{ll}
                             2 & \text{ if $g=1$ } \\
                             -2 & \text{ if $g=-1$ } \\
                             0 & \text{ otherwise } \end{array} \right. \]
Since $-1$ is central in $SL_2(\IF_3)$, this representation is $\Ff_{\alpha}$-invariant. Therefore
\[ R(\Ff_{\alpha};\IZ/2) \cong \IZ \rho \]

We need to determine its structure of module over $R(\Ff)$. The group $S$ is normal in $A_4$ and $\Aut_{A_4}(S) \cong \IZ/3$
is generated by conjugation by $(1,2,3)$. The set of irreducible representations of $S$ is given by $\{1, 1 \otimes \gamma, 
\gamma \otimes 1, \gamma \otimes \gamma \}$, where $\gamma$ is the sign representation of $\IZ/2$. The action of 
$\Aut_{A_4}(S)$ on the three nontrivial irreducible representations is transitive and so the representations
$1$ and $x = 1 \otimes \gamma + \gamma \otimes 1 + \gamma \otimes \gamma$ form a basis of $R(\Ff)$. One can 
check that $x^2 = 2x+3$ and so
\[ R(\Ff) \cong \IZ[x]/(x^2-2x-3) \]
The composition of the representation $x$ with the quotient $S_{\alpha} \to S$ results in the representation of 
$S_{\alpha}$ with character
\[ \chi(g) = \left\{ \begin{array}{ll}
                             3 & \text{ if $g=1$, $-1$ } \\
                             -1 & \text{ otherwise } \end{array} \right. \]
and computing their characters, we conclude that $x \cdot \rho =3\rho$. In particular, 
the augmentation ideal $I(\Ff)$ acts trivially on $R(\Ff_{\alpha};\IZ/2)$ and therefore
\[ ^{\alpha} K( (BA_4) ^{\wedge}_2 ) \cong R(\Ff_{\alpha};\IZ/2) \cong \IZ \]
We can also compute $K( (BA_4) ^{\wedge}_2 ) \cong \IZ[[y]]/(y^2+4y)$ and since $y$ is the 
element in $K$-theory corresponding to the element from $x-3$ in $R(\Ff)$, we see that $y$ 
acts trivially on $^{\alpha} K( (BA_4) ^{\wedge}_2 )$.
\end{example}

\begin{example}
In \cite{RV}, Ruiz and Viruel classified up to equivalence the saturated fusion systems over 
\[ S = \langle a, b, c \mid a^7 = 1, b^7 = 1, c^7 = 1, [a,b]=c, [a,c]=1, [b,c] = 1 \rangle , \]
the extraspecial $7$-group of order $7^3$ and exponent $7$, finding three exotic $7$-local finite groups.
Following the notation in \cite{Y}, we use the names $RV_1$, $RV_2$ and $RV_3$ for these exotic $7$-local 
finite groups and their fusion systems, and $BRV_i$ for their classifying spaces.

Theorem 5.5.4 in \cite{Go} describes the irreducible representations of $S$. It has $49$ one-dimensional representations 
$x^iy^j$ which come from the quotient $ S \to S/Z(S) \cong \IZ/7 \times \IZ/7 $. Let $\omega$ denote a primitive seventh
root of unity. The six remaining irreducible representations $z_j$ for $j=1,\ldots,6$ are determined by:
\[ z_j(a) = \diag (\omega^{6j},\omega^{5j},\ldots,\omega^j,1) \]
and $z_j(b)$ is the linear transformation that sends $e_1$ to $e_7$ and $e_n$ to $e_{n-1}$ if $n \geq 2$. Here 
$\{ e_1, \ldots, e_n \}$ denotes the standard basis of $\IC^7$. Note that $z_j(c) = w^j I $. 

In $RV_1$, we have $\Out_{RV_1}(S) \cong \IZ/36 \rtimes \IZ/2$ and this group is
generated by the three automorphisms defined by
\[ \left\{ \begin{array}{lll} 
   a & \mapsto & a^3 \\
   b & \mapsto  & b \end{array} \right. \qquad \left\{ \begin{array}{lll} 
   a & \mapsto & a \\
   b & \mapsto  & b^3 \end{array} \right. \qquad \left\{ \begin{array}{lll} 
   a & \mapsto & b \\
   b & \mapsto  & a \end{array} \right. \]
Moreover, the elementary abelian subgroup generated by $a$ and $c$ has $GL_2(\IF_7)$ as group of automorphisms.
Using these automorphisms, it is easy to check that the representation ring of $RV_1$ is generated by the trivial
representation and the representation 
\[ U = 7 \left( \sum_{k=1}^6 z_k \right) + \sum_{(i,j) \neq (0,0)} x^i y^j = \reg_S - 1 \]
Therefore 
\[ R(RV_1) \cong \IZ[U]/(U^2-341U-342) \]
On the other hand, we have $\Out_{RV_2}(S) \cong \IZ/3 \times SD_{16}$, generated by the
automorphisms determined by  
\[ \left\{ \begin{array}{lll} 
   a & \mapsto & a^{-1} \\
   b & \mapsto  & b \end{array} \right. \qquad \left\{ \begin{array}{lll} 
   a & \mapsto & a^2 \\
   b & \mapsto  & b^2 \end{array} \right. \qquad \left\{ \begin{array}{lll} 
   a & \mapsto & b^{-1} \\
   b & \mapsto  & a \end{array} \right. \qquad \left\{ \begin{array}{lll} 
   a & \mapsto & a^{-1}b^{-1} \\
   b & \mapsto  & ab^{-1} \end{array} \right. \]
The elementary abelian subgroup generated by $c$ and $ab^2$ and the one generated by $c$ and $a$ both have $SL_2(\IF_7) \rtimes \IZ/2$ as group of automorphisms. 
Using all these automorphisms, it is again easy to check that the representation ring of $RV_2$
is generated by $1$ and $ U $. So its representation ring coincides with $R(RV_1)$. Since the character of $U$ is equal to $-1$ for all nontrivial elements of $S$ and $RV_2$ is a subcategory of $RV_3$, we have $R(RV_2)=R(RV_3)$. Therefore we conclude 
\[ K(BRV_i) \cong \IZ[[u]]/(u^2+343u) \]
for $i=1,2,3$, where $u$ is the element in $K$-theory coming from $U-342$. Similarly, $K(BRV_i;\IZ ^{\wedge}_7) \cong \IZ ^{\wedge}_7[[u]]/(u^2+343u)$.
\end{example}

\bibliography{VersionArxivBib2}
\bibliographystyle{amsplain}

\end{document}